\documentclass[11pt]{article}

\usepackage{amsmath,amsthm,mathrsfs}
\usepackage{amssymb,latexsym}

\usepackage{enumerate}

\usepackage[mathscr]{euscript}

\newcommand{\BB}{{\mathcal B}}

\newcommand{\EE}{{\mathcal E}}
\newcommand{\FF}{{\mathcal F}}

\newcommand{\MM}{{\mathcal M}}

\newcommand{\SM}{{\mathcal{S}}}

\newcommand{\BD}{{\mathbb D}}

\newcommand{\BM}{{\mathbb M}}

\newcommand{\BR}{{\mathbb R}}

\newcommand{\dyw}{\mbox{\rm div}}
\newcommand{\fch}{{\mathbf{1}}}

\newtheorem{theorem}{\bf Theorem}[section]
\newtheorem{proposition}[theorem]{\bf Proposition}
\newtheorem{lemma}[theorem]{\bf Lemma}
\newtheorem{corollary}[theorem]{\bf Corollary}

\theoremstyle{definition}
\newtheorem{definition}[theorem]{Definition}

\newtheorem{remark}[theorem]{Remark}

\numberwithin{equation}{section}

\begin{document}

\title {Robin problem with measure data and singular nonlinearities on the boundary}
\author {Andrzej Rozkosz\\
{\small Faculty of Mathematics and Computer Science,
Nicolaus Copernicus University} \\
{\small  Chopina 12/18, 87--100 Toru\'n, Poland}\\
e-mail: rozkosz@mat.umk.pl}
\date{}
\maketitle
\begin{abstract}
We consider the Robin problem for a uniformly elliptic divergence operator with measure data on the right-hand side of the equation and an absorption term on the boundary involving blowing up terms.
We prove the existence of a  positive renormalized solution and provide it stochastic representation, which can be viwed as a  generalized nonlinear Feynman--Kac formula. From this representation we derive some additional regularity results for the solution and some uniqueness results.
\end{abstract}

\section{Introduction}
\label{sec1}

Let $D$ be a bounded Lipschitz domain in $\BR^d$, $d\ge3$. We are interested in existence, uniqueness and regularity of positive solutions of the nonlinear Robin problem (also called Fourier problem or third boundary-value problem)
\begin{equation}
\label{eq1.1}
-L u=\mu\quad\mbox{in }D,\qquad-\frac{\partial u}{\partial{\gamma_a}}
+\beta\cdot u=h\cdot g(u)\quad \mbox{on }\partial D,
\end{equation}
where  $L$ is a uniformly elliptic divergence form operator  defined by
\begin{equation}
\label{eq1.10}
L=\sum^{d}_{i,j=1}\frac{\partial}{\partial x_i}
\Big(a_{ij}(x)\frac{\partial}{\partial x_j}\Big)
\end{equation}
for some symmetric matrix-valued function $a:\bar D\rightarrow\BR^d\otimes\BR^d$ and $\gamma_a=a\cdot\mathbf{n}(x)$, where $\mathbf{n}(x)$
is the unit inward normal vector  at $x\in \partial D$.
As for the data, we assume that  $\mu$ is a smooth measure on $D$ (with respect to the capacity associated with $L$),  
$h:\partial D\rightarrow\BR$ is a Borel  measurable function such that
\begin{equation}
\label{eq1.2}
h\ge0\,\,\, \sigma\mbox{-a.e.},\quad \mu(D)+\|h\|_{L^1(\partial D;\sigma)}<\infty,\quad \mu(D)+\|h\|_{L^{1}(\partial D;\sigma)}>0,
\end{equation}
where 
$\sigma$ is the surface measure on $\partial D$,
and $\beta:\partial D\rightarrow\BR$ is a Borel measurable function such that
\begin{equation}
\label{eq1.3}
\beta\in L^{\infty}(\partial D;\sigma),\quad  \beta\ge0\quad\sigma\mbox{-a.e.},\quad \|\beta\|_{L^1(\partial D;\sigma)}>0.
\end{equation}
We assume that in (\ref{eq1.1}),  in the nonlinear term, $g:(0,\infty)\rightarrow[0,\infty)$  is a continuous function such that
\begin{equation}
\label{eq1.4}
c_1\le g(y)y^{\gamma}\le c_2,\quad y>0,
\end{equation}
for some constants $c_1,c_2,\gamma>0$. Therefore  $h\cdot g(u)$ forms a blowing up term.
In fact, we also deal with   mixed nonlinearities of the form $hg^1(u)+hg^2(u)$, where $g^1,g^1$ are continuous and for some constants $\gamma_1,\gamma_2>0$ satisfy
\begin{equation}
\label{eq1.7}
c_1\le g^1(y)y^{-\gamma_1}\le c_2,\quad c_1\le g^2(y)y^{-\gamma_2}\le c_2,\quad y>0.
\end{equation}

Problem of the form (\ref{eq1.1}) was  recently addressed in \cite{DOS}
for equation with $L^1$-data, i.e. when $\mu(dx)=f(x)\,dx$ for some $f\in L^1(D;dx)$.  Let us stress, however, that in \cite{DOS} the authors also consider equations with more general than $L$ operators of Leray--Lions type
and with the term $\beta\cdot u$ replaced by $\beta q(u)$ with some nonnegative nonlinear  $q$ (for instance, for $L$  given by (\ref{eq1.10}) it is allowed that $q$ is continuous  such that $q(0)=0$, $q(y)\ge y$, $y\ge0$). Robin problem with $L^1$ or measure data and different boundary conditions (with no blowing up terms) are  considered for instance in \cite{AMST,GO,Pr}. For recent results on Robin problem with homogeneous boundary data but nonlinear terms on the right-hand side of the equation (involving the gradient of the solution as well) we refer the reader to \cite{DP,GMM} and the refrences therein.

Our work is motivated by \cite{DOS} and is strongly related to this paper. It is also strongly related  to \cite{K:JFA} (and to some extent to \cite{BO}) although in \cite{K:JFA} the Robin problem is not mentioned and links between (\ref{eq1.1}) and the problem considered in \cite{K:JFA} are not self-evident at all.
Below we briefly describe our main results and approach to (\ref{eq1.1}) in the case where $L=\Delta$, i.e. when  $a$ is the $d$-dimensional identity matrix. Let  $\BD$ denote the classical Dirichlet form on $L^2(\bar D)$ defined by
\[
\BD(u,v)=\int_D\nabla u\cdot\nabla v\,dx,\quad u,v\in D(\BD):=H^1(D),
\]
and let $\BD^{\beta\cdot\sigma}$ be the form $\BD$ perturbed by the measure $\beta\cdot\sigma$, i.e.
\[
\BD^{\beta\cdot\sigma}(u,v)=\BD(u,v)+\int_{\partial D}uv\beta\,d\sigma, \qquad u,v\in D(\BD^{\beta\cdot\sigma}):=H^1(D)\cap L^2(\bar D;\beta\cdot\sigma).
\]
Both $\BD$ and $\BD^{\beta\cdot\sigma}$ are regular Dirichlet forms on $L^2(\bar D;dx)$.  The generator of $\BD$, which we denote by $\Delta_N$, may be called the  Laplace operator with Neumann boundary conditions. The generator of the perturbed form can be formally written as $\Delta_N-\beta\cdot\sigma$. As a matter of fact, in the present paper we regard (\ref{eq1.1}) as equation of he form
\begin{equation}
\label{eq1.5}
(-\Delta_N+\beta\cdot\sigma)u=\mu+hg(u)\cdot\sigma\quad\mbox{in }\bar D,
\end{equation}
i.e. a Schr\"odinger equation for $\Delta_N$ with potential $\beta\cdot\sigma$ and a nonliner term involving measures on the right-hand side. The reason is that properly defined weak solutions of (\ref{eq1.5}) are weak solutions of (\ref{eq1.1})
and vice versa, and similarly, renormalized solutions of (\ref{eq1.1}) are renormalized solutions of (\ref{eq1.5}). On the other hand, equation (\ref{eq1.5}) falls within the scope of semilinear equations with Dirichlet operators and smooth  measure data, and therefore can be effectively studied by probabilistic methods. In fact, in our proofs we apply or modify  methods used before in \cite{K:JFA,KR:NoD,KR:PA}. It is worth noting here that the transfer of (\ref{eq1.1}) to equation (\ref{eq1.5}) involving the  measure $\sigma$ is useful even in case  $\mu(dx)=f\,dx$ for some nonnegative $f$. A similar idea of transfering  the study of  a Robin problem to the study, by probabilistic methods, of
a related problem with measure data was used in \cite{RS}. Another examples of effective  applications of probabilistic methods to the study of Robin problems are found in \cite{DP,M}.

Our main result says that under the assumptions (\ref{eq1.2}), (\ref{eq1.3}) and (\ref{eq1.4}) or (\ref{eq1.7}) there exists a positive renormalized solution $u$ of (\ref{eq1.1}) such that $hg(u)\in L^1(\partial D;\sigma)$. If moreover $g$ is nondecreasing, then it is unique. We also give some regularity results for $u$. For instance, we prove that
for every $k>0$, $T_k(u):=(u\wedge k)\in H^1(D)$ and
\begin{equation}
\label{eq1.8}
\EE^{\kappa}(T_k(u),T_k(u))\le k(\mu(D)+\|hg(u)\|_{L^1(\partial D;\sigma)}).
\end{equation}
If, in addition, $\mu=0$, then $u^{(\gamma+1)/2}\in H^1(D)$ and for some constant $c(\gamma)>0$ we have
\begin{equation}
\label{eq1.9}
\EE^{\kappa}(u^{(1+\gamma)/2}, u^{(1+\gamma)/2})
\le c(\gamma)\|h\|_{L^1(\partial D;\sigma)}.
\end{equation}
As already mentioned, as compared with \cite{DOS},  the  novelty is that we consider general smooth bounded measures $\mu$ and not only $L^1$-data. For that reason we consider renormalized solutions and not entropy solutions as in \cite{DOS}. Furthermore, our approach allows us to treat in a unified way equations with  mixed nonlinerities of the form (\ref{eq1.7}).

A remarkable feature of our approach is that we find solutions of (\ref{eq1.1}) that are quasi continuous with respect to the capacity associated with  $\BD$ and satisfy quasi everywhere the integral equation
\begin{equation}
\label{eq1.6}
u(x)=\int_DG^{\beta\cdot\sigma}(x,y)\,\mu(dy)+\int_{\partial D}G^{\beta\cdot\sigma}(x,y)h(y)g(u(y))\,\sigma(dy),
\quad x\in\bar D,
\end{equation}
where $G^{\beta\cdot\sigma}$ is the Green function on $\bar D$ for the operator $\Delta_N-\beta\cdot\sigma$. By using the probabilistc potential theory one can show that the above equation can be written in probabilistic terms involving the process $\BM$ associated with $\BD$ (reflected Brownian motion in $\bar D$) and additive functionals of $\BM$ in the Revuz correspondence with the measures $\mu $ and $\sigma$ (see Section \ref{sec3}). This probabilistic version of (\ref{eq1.6})  may be viewed as a nonlinear, generalized Feynman--Kac formula, which in turn provides a link between (\ref{eq1.6}) and some backward stochastic differential equations (BSDEs).  In fact we  use this link and methods from the theory of BSDEs to get (\ref{eq1.6}). Once the existence of $u$ satisfying (\ref{eq1.6}) is shown,  no further work is required. From the probabilistic counterpart of (\ref{eq1.6}) and general results proved in \cite{K:JFA,KR:JFA,KR:NoD} it follows that $u$ is a renormalized solution and  satisfies (\ref{eq1.8}), (\ref{eq1.9}). From (\ref{eq1.8}) and known results proved in \cite{BBGGPV} it also follows that $u,\nabla u$  belong to  Marcinkiewicz spaces of orders $r$ determined by $d$. As another simple application of links between (\ref{eq1.6}) and BSDEs we show that if $g$ is nonincreasing, $u_n$ is the unique renormalized solution of (\ref{eq1.1}) with $\mu_n(dx)=f_n(x)\,dx$ and $f_n\rightarrow f$ in $L^1(D)$, then $u_n\rightarrow u$ quasi everywhere, where $u$ is the renormalized solution of (\ref{eq1.1}) with $\mu(dx)=f(x)\,dx$.

In fact, the results described above hold true for the form $\BD$ replaced by the form associated with $L$. An additional work is required to show that  renormalized solutions of (\ref{eq1.1}) are weak  solutions. We are able to show this is when $a$ is smooth,  $\mu(dx)=f(x)\,dx$ with nontrivial $f$ and $f\in L^p(D)$, $h\in L^q(\partial D;\sigma)$ with suitably chosen $p,q$.

\section{Preliminaries}

In the paper, $D$ is a bounded domain in $\BR^d$, $d\ge3$, with  Lipschitz  boundary,
$\bar D=D\cup\partial D$. We denote by  $\partial$ an isolated point adjoint to $\bar D$ and we adopt the convention that every function $f$ on $\bar D$ is extended to $\bar D\cup\{\partial\}$ by setting $f(\partial)=0$.
For $E\subset \bar D$,  $\BB^+(E)$ is the set of all nonnegatve Borel measurable functions on $E$.

We denote by $m$ or $dx$  the Lebesgue measure on $D$ and by $\sigma$ the surface measure on $\partial D$. For $E\subset\bar D$, a measure $\nu$ on $E$ and  a function $f:E\rightarrow\BR$  we use occasionally the notation
\[
\langle \nu,f\rangle=\int_Ef(x)\,\nu(dx)
\]
whenever the integral is well defined.

\subsection{Dirichlet forms and perturbed Dirichlet forms}

Let $I_d$ denote the $d$-dimensional identity matrix. In the paper,   $a:\bar D\rightarrow\BR^d\otimes\BR^d$ is a measurable, symmetric matrix-valued function which for some $\Lambda\ge1$ satisfies the condition
\begin{equation}
\label{eq2.7}
\Lambda^{-1} I_d\le a(\cdot)\le \Lambda I_d
\end{equation}
in the sense of nonnegatve definite matrices.
We consider the  form
\[
\EE(u,v)=\sum^d_{i,j}\int_Da_{ij}(x)\frac{\partial u}{\partial x_i}(x)
\frac{\partial v}{\partial x_j}(x)\,dx,\quad u,v\in D(\EE):=H^1(D),
\]
and for $\lambda\ge0$ we set $\EE_{\lambda}(u,v)=\EE(u,v)+\lambda(u,v)$, where $(\cdot,\cdot)$ is the standard inner product in $L^2(D;m)$ and $H^1(D)$ is the usual Sobolev space of order 1. By the results proved in \cite{FT1, FT2} (for more general domains), $\EE$ is a regular Dirichlet form on $L^2(\bar D;m)$.

It is known (see \cite{FT1,FT2} or \cite[Example 5.2.2]{FOT} for the case $a=I_d$) that $\sigma$ is a smooth measure relative to $\EE$.
To shorten notation, for $\beta\in L^{\infty}(\partial D;\sigma)$ we let stand $\kappa$ for the measure $\beta\cdot\sigma$, i.e. $\kappa$ is  defined for Borel subsets of $\bar D$ by  $\kappa(B)=\int_{B\cap\partial D}\beta(x)\,\sigma(dx)$.
We denote  by  $(\EE^{\kappa},D(\EE^{\kappa}))$ the form perturbed by $\kappa$, that is
\[
\EE^{\kappa}(u,v)=\EE(u,v)+\int_{\partial D}uv\beta\,d\sigma, \qquad u,v\in D(\EE^{\kappa}):=H^1(D)\cap L^2(\bar D;\kappa).
\]
By \cite[Lemma 6.1.1]{FOT},  $(\EE^{\kappa},D(\EE^{\kappa}))$ is a regular Dirichlet form on $L^2(\bar D;m)$. By the classical trace theorem (see, e.g., \cite[Chapter 2, Theorem 4.2]{N}) if $u\in H^1(D)$, then $u\in L^2(\bar D;\fch_{\partial D}\cdot\sigma)$, and hence $u\in L^2(\bar D;\kappa)$ if $\beta\in L^{\infty}(\partial D;\sigma)$. Thus, for essentially bounded  nonnegative $\beta$ we have $D(\EE^{\kappa})=H^1(D)$.

Suppose that $\beta$ satisfies (\ref{eq1.3}). For $u\in H^1(D)$ set $\|u\|^2_{\kappa}=\EE^{\kappa}(u,u)$. Then
\begin{equation}
\label{eq2.1}
c_1\|u\|_{H^1(D)}\le\|u\|_{\kappa}\le c_2\|u\|_{H^1(D)}
\end{equation}
for some strictly positive constants $c_1,c_2$. For the first inequality see, e.g., \cite[Appendix I, Lemma 4.1]{G}. The second inequality follows from the trace theorem.
Notice that  the first inequality in (\ref{eq2.1}) implies that the form $\EE^{\kappa}$ is transient because for
$u\in H^1(D)$ we have
\[
\int_D|u|\,dx\le (m(D))^{1/2}\|u\|_{L^2(D;m)}\le (m(D))^{1/2}\|u\|_{H^1(D)}\le c (\EE^{\kappa}(u,u))^{1/2},
\]
i.e. condition \cite[(1.5.6)]{FOT} is satisfied with the reference function equal to 1.
A probabilistic argument showing  transience of $\EE^{\kappa}$ is given right after (\ref{eq2.6}). Moreover, (\ref{eq2.1}) implies that $(\EE^{\kappa},H^1(D))$   is the same as its extended Dirichlet space.

In the present paper, we define quasi notions (capacity, quasi continuity, exceptional sets)  as in \cite[Section 2.1]{FOT} with respect to the form $(\EE,H^1(D))$. We will say that a property of points in $\bar D$ holds quasi everywhere (q.e. in abbreviation) if it holds outside some exceptional set.

We denote by $\SM$ the set of all smooth measures on $\bar D$ relative to the form $\EE$. By \cite[Lemma 6.1.2]{FOT}, $\SM=\SM^{\kappa}$, where $\SM^{\kappa}$ denotes the set of all smooth measures on $\bar D$ relative to $\EE^{\kappa}$.

In the whole paper we assume  that $\mu$ is a Borel measure on $\bar D$ such $\mu(\bar D)=\mu(D)<\infty$ and $\mu\in\SM$. By \cite[(4.4.14)]{FOT}, the form $(\EE,H^1_0(D))$ can be considered as the part of the form $(\EE,H^1(D))$ on $D$, so  from \cite[Theorem 4.4.3(ii)]{FOT} it follows that $\mu$ is a somooth measure with respect to the capacity associated with $(\EE,H^1_0(D))$. Therefore from \cite[Theorem 2.1]{BGO} (or \cite[Example 4.1]{KR:BPAN})
it follows that $\mu$ has the representation $\mu(dx)=f-\dyw(F)$ for some $f\in L^1(D;m)$ and $F=(F_1,\dots,F_d)$ with $F_i\in L^2(D;m)$, $i=1,\dots,d$.

A Radon measure $\nu$ on $\bar D$ is said to be of finite 0-energy integral relative to the form $\EE^{\kappa}$ ($\nu\in S^{(0),\kappa}_0$ in notation) if for some $c>0$,
\[
\langle \nu,v\rangle\le c \|v\|_{\kappa},\quad v\in D(\EE^{\kappa})\cap C(\bar D).
\]
By \cite[Lemma 2.2.3, Eq. (2.2.22)]{FOT}, $S^{(0),\kappa}_0\subset\SM^{\kappa}$.
Note also that $\sigma\in S^{(0),\kappa}_0$ because by the trace theorem and (\ref{eq2.1}),
\[
\int_{\partial D}v\,d\sigma\le C\|v\|_{L^2(\partial D;\sigma)}\le  C_1\|v\|_{H^1(D)}\le C_2\|v\|_{\kappa}.
\]

If $\nu\in S^{(0),\kappa}_0$, then by \cite[Theorem 2.2.5]{FOT}
there exists a function $U^{\kappa}\nu\in H^1(D)$, called the potential of $\nu$ (relative to $\EE^{\kappa}$), such that
\begin{equation}
\label{eq2.5}
\EE^{\kappa}(U^{\kappa}\nu,v)=\langle\nu,\tilde v\rangle,\quad v\in H^1(D),
\end{equation}
where $\tilde v$ is a quasi continuous version of $v$.

\subsection{Markov processes}

It is known (see \cite{FT1, FT2}) that there exists a conservative diffusion process $\BM=\{(X,(\FF_t)_{t\ge0},P_x),x\in\bar D\}$ on $\bar D$ associated with the Dirichlet form $\EE$. We denote by $E_x$ the expectation with respect to $P_x$. When $a=(1/2)I_d$, then $\BM$ is nothing else but  a reflecting standard Brownian motion (see \cite[Exercises 4.5.3, 5.2.2]{FOT}). Let  $p_t$ be the transition kernel of $\BM$, i.e. $p_t(x,B)=P_x(X_t\in B)$, $t>0,x\in\bar D$, $B\in\BB(\bar D)$, and  $(P_t)$ be the semigroup
associated with  $\BM$, i.e.  $P_tf(x)=\int_{\bar D}f(y)p_t(x,dy)=E_xf(X_t)$ for $f\in\BB_b(\bar D)$ and  $t>0$, $x\in\bar D$. It is  known (see \cite{FT1, FT2}) that  the semigropup $(P_t)$ is strong Feller and $p_t$ satisfies the following absolute continuity condition: $p_t(x,\cdot)\ll m$ for any $t>0$ and $x\in\bar D$.

For $\nu\in\SM$ let  $A^{\nu}$ denote the positive continuous additive functional (PCAF) of $\BM$ in the Revuz correspondence with $\nu$.
By \cite[Lemma 5.1, Theorem 5.1]{FT2},  the surface measure $\sigma$ belongs to the space of smooth measures in the strict sense (relative to $\EE$), and hence, by \cite[Theorem 5.1.7]{FOT}, there is  a unique PCAF $A^{\sigma}$ of $\BM$ in the strict sense with the Revuz measure $\sigma$.
In the paper, we denote by $A$  the PCAF of $\BM$ in the Revuz correspondence with $\kappa$, and by $B$ the PCAF of $\BM$ with the Revuz measure $\mu$,  i.e.
\[
A_t=A^{\kappa}_t=(\beta\cdot A^{\sigma})_t=\int^t_0\beta(X_s)\,dA^{\sigma}_s,\qquad B_t=A^{\mu}_t,\quad t\ge0.
\]
Let $\BM^A=\{(X^A,(\FF^A_t)_{t\ge0},P^A_x), x\in\bar D\}$ be the canonical subprocess of $\BM$  with respect to the multiplicative functional $(e^{-A_t})$ (see \cite[Theorem A.2.11]{FOT}). We denote by $\zeta^A$ its life time and by $E^A_x$ the expectation with respect to $P^A_x$.  The transition semigroup associated with $\BM^A$ has the form
\begin{equation}
\label{eq2.6}
P_t^Af(x)=E^Af(X^A_t)=E_x(e^{-A_t}f(X_t)), \quad f\in\BB_b(\bar D),\quad t>0,\,x\in\bar D.
\end{equation}
If $\beta$ satisfies (\ref{eq1.3}), then  $\kappa$ is nontrivial, which implies that for every $x\in\bar D$ there is $t>0$ such that   $P^A_t1(x)=E_xe^{-A_t}<1$. This and \cite[Lemma 1.6.5]{FOT} imply that the semigroup $(P^A_t)$ is transient, and hence, by \cite[Theorem 1.5.1]{FOT}, that $\EE^{\kappa}$ is transient 

We denote by $(R^A_{\alpha}f)_{\alpha>0}$ the resolvent associated with $\BM^A$ (or the semigroup $(P^A_t)$), i.e.
\[	
R^A_{\alpha}f(x)=E_x\int^{\infty}_0e^{-\alpha t}P^A_tf(x)\,dt
=E\int^{\infty}_0e^{-\alpha t-A_t}f(X_t)\,dt,\quad f\in\BB^+(\bar D),\quad x\in\bar D.
\]
We also set $Rf=R_0f$. By \cite[Exercise 6.1.1]{FOT}, $\BM^A$ satisfies the absolute continuity condition and hence, by \cite[Lemma 4.2.4]{FOT}, for every $\alpha>0$ there is a nonnegative jointly measurable symmetric function $G^{\kappa}_{\alpha}(x,y)$, $x,y\in\bar D$  which is $\alpha$-excessive (relative to $\BM^A$) in each variable and has the property that for any $f\in\BB^+(\bar D)$ and $x\in\bar D$,
\begin{equation}
\label{eq2.2}	
R^A_{\alpha}f(x)=G^{\kappa}_{\alpha}f(x),\quad\mbox{where}\quad
G^{\kappa}_{\alpha}f(x)=\int_{\bar D}G^{\kappa}_{\alpha}(x,y)f(y)\,dy
\end{equation}
Furthermore, $G^{\kappa}_{\alpha}(x,y)$ increases as $\alpha\downarrow0$. We set $G^{\kappa}(x,y)=\lim_{\alpha\downarrow0}G^{\kappa}_{\alpha}(x,y)$. The function  $G^{\kappa}$ is jointly measurable,  symmetric and $0$-excessive in each variable. By monotone convergence, letting $\alpha\downarrow0$ in (\ref{eq2.2}) we get 
\begin{equation}
\label{eq2.3}
R^Af(x)=G^{\kappa}f(x),\quad f\in\BB^+(\bar D),\quad x\in\bar D,
\end{equation}
where
\begin{equation}
\label{eq2.4}
R^Af(x)=E_x\int^{\infty}_0e^{-A_t}f(X_t)\,dt,\qquad G^{\kappa}f(x)
=\int_DG^{\kappa}(x,y)f(y)\,dy.
\end{equation}
In the terminology used in \cite{FOT},  $G^{\kappa}(x,y)$ is  called the resolvent density of  $\BM^A$ (and is denoted by $r^A(x,y)$; see \cite[Exercise 6.1.1]{FOT}).
We denote it by $G^{\kappa}(x,y)$ because in the situation considered in the present paper it is the Green function on $\bar D$ for the operator associated with $\EE^{\kappa}$.

\section{Probabilistic  and integral solutions}
\label{sec3}

In what follows we assume that $f,h, g,\beta$ are Borel measurable and nonnegative.

\subsection{Basic definitions and auxiliary results}
\begin{definition}
A function $u$ defined q.e. on $\bar D$ is a probabilistic solution of (\ref{eq1.1}) if $u$ is quasi continuous, $u>0$ q.e. in $\bar D$ and for q.e. $x\in\bar D$,
\begin{equation}
\label{eq3.2}
u(x)=E^A_x\int^{\infty}_0dB_t
+E^A_x\int^{\infty}_0h(X^A_t)g(u(X^A_t))\,dA^{\sigma}_t.
\end{equation}
\end{definition}

\begin{remark}
\label{rem3.2}
(i) Since every quasi continuous function is finite q.e. (this follows for instance from
\cite[Theorem 2.1.2]{FOT}),  (\ref{eq3.2}) means in particular that both integrals on the right-hand side  are finite for q.e. $x\in \bar D$.
\smallskip\\
(ii) $B$ can be regarded as a PCAF of $\BM^A$ with the Revuz measure $\mu$ again (see \cite[Theorem (2.22)]{FG} and also \cite[p. 331]{FOT}).
In particular, $B_t=B_{t\wedge\zeta^A}$, $t\ge0$, $P_x$-a.s. for q.e. $x\in\bar D$. Similarly, if $\nu=\mu+h\cdot\sigma$, then $A^{\nu}$ can be regarded as a PCAF of $\BM^A$ with the Revuz measure $\nu$.
\smallskip\\
(iii) Let $\nu=\mu+\sigma$, $\psi(x,y)=\fch_{D}(x)+\fch_{\partial D}h(x)g(y)$. Then (\ref{eq3.2}) can be written in the form
\[
u(x)=E^A_x\int^{\infty}_0\psi(X^A_t,u(X^A_t))\,dA^{\nu}_t
=E^A_x\int^{\zeta^A}_0\psi(X^A_t,u(X^A_t))\,dA^{\nu}_t.
\]
\end{remark}

\begin{definition}
A function $u$ defined q.e. on $\bar D$ is an integral solution of (\ref{eq1.1}) if $u$ is quasi continuous, $u>0$ q.e. in $\bar D$ and (\ref{eq1.6}) is satisfied for q.e. $x\in\bar D$.
\end{definition}

We shall see that probabilistic solutions are integral solutions in the sense that (\ref{eq1.6}) is satisfied. To see this we need a  result extending (\ref{eq2.3}). To state it, for  $\nu\in\SM$  we set
\[
G^{\kappa}\nu(x)=\int_{\bar D}G^{\kappa}(x,y)\,\,\nu(dy),\quad x\in\bar D.
\]
This notation is in agreement with our previous notation (\ref{eq2.4}).
\begin{lemma}
\label{lem3.5}
If $\nu\in S^{(0),\kappa}_0$, then  $G^{\kappa}\nu$ is an $m$-version of the potential $U^{\kappa}\nu$.
\end{lemma}
\begin{proof}
The measure $\nu$ is of finite energy (relative to $\EE^\kappa$) in the sense defined in \cite[Section 2.2]{FOT}. For $\alpha>0$ let $U^{\kappa}_{\alpha}\nu$ denote its $\alpha$-potential. By \cite[Exercise 4.2.2]{FOT} applied to the form $\EE^{\kappa}$, for every $\alpha>0$ the function $G^{\kappa}_{\alpha}\nu$ is an $m$-version of the potential $U^{\kappa}_{\alpha}\nu$.
By monotone convergence, $G^{\kappa}_{\alpha}\nu\rightarrow G^{\kappa}\nu$ pointwise  and by \cite[Lemma 2.2.11]{FOT}, $\|U^{\kappa}_{\alpha}\nu-U^{\kappa}\nu\|_{\kappa}\rightarrow0$  as $\alpha\downarrow0$, which gives the desired result.
\end{proof}

\begin{lemma}
\label{lem3.6}
Let $\nu\in\SM$. Then $E^A_xA^{\nu}_{\infty}=G^{\kappa}\nu(x)$ for q.e. $x\in\bar D$. If, in addition, $\nu(\bar D)<\infty$, then $G^{\kappa}\nu(x)<\infty$ for q.e. $x\in\bar D$.
\end{lemma}
\begin{proof}
By the 0-order version of \cite[Theorem 2.2.4]{FOT}, there exists an $\EE^{\kappa}$-nest $\{F_k\}$ such that $\nu_k:=\fch_{F_k}\nu\in S^{(0),\kappa}_0$ for every $k\ge1$. From the proof of \cite[Lemma 6.1.1]{CF} (applied to $\EE^{\kappa}$) it follows that for each $k\ge1$ the function $x\mapsto E^A_xA^{\nu_k}$ is a quasi continuous version of $U^{\kappa}\nu_k$. By this and Lemma \ref{lem3.5},  $E^A_xA^{\nu_k}=G^{\kappa}\nu_k$, $k\ge1$ for q.e. $x\in\bar D$. Letting $k\rightarrow\infty$ shows the first assertion.
Since $\EE^{\kappa}$ is transient, the second assertion follows from the first  one and \cite[Lemma 4.3, Proposition 5.13]{KR:JFA}.
\end{proof}

\begin{proposition}
\label{prop3.3}	
A function $u$ is a probabilistic solution of \mbox{\rm(\ref{eq1.1})} if and only if it is an integral solution of \mbox{\rm(\ref{eq1.1})}.
\end{proposition}
\begin{proof}
Let $\varphi\in\BB^+(\bar D)$. Then for q.e. $x\in\bar D$ we have
\begin{equation}
\label{eq3.14}
E^A_xB_{\infty}=G^{\kappa}\mu(x),\qquad E^A_x\int^{\infty}_0\varphi(X^A_t)\,dA^{\sigma}_t
=G^{\kappa}(\varphi\cdot\sigma)(x).
\end{equation}
Indeed, the first equality follows from Lemma \ref{lem3.6} applied to $\nu:=\mu$. To get  the second one we apply Lemma \ref{lem3.6}  to $\nu_n=(\varphi\wedge n)\cdot\sigma$ and then pass to the limit as $n\rightarrow\infty$. Clearly (\ref{eq3.14}) implies the desired result.
\end{proof}


\begin{lemma}
\label{lem3.4}
If $\mu,h,\beta$ satisfy \mbox{\rm(\ref{eq1.2}), (\ref{eq1.3})} then  $G^{\kappa}(\mu+h\cdot\sigma)(x)>0$
for q.e. $x\in\bar D$.
\end{lemma}
\begin{proof}
Write $\nu=\mu+h\cdot\sigma$. Since $\BM$ is strong Feller, it is irreducible by \cite[Exercise 4.6.3]{FOT}. This and  \cite[Theorem 4.7.1]{FOT} imply that  condition $(\EE.7)$ from \cite{K:JFA} is satisfied. Hence,  by  \cite[Remark 3.5]{K:JFA}, $E_xA^{\nu}_{\infty}>0$ for q.e. $x\in\bar D$.
From Lemma \ref{lem3.6} it follows now that  $E^A_xA^{\nu}_{\infty}>0$ for q.e. $x\in\bar D$ because by Remark \ref{rem3.2}(ii), $\nu$ is the  Revuz measure of $A^{\nu}$ considered as the PCAF of $\BM^A$.
\end{proof}

Probabilistic solutions of (\ref{eq1.1}) are closely related to the backward stochastic differential equations (BSDEs) of the form
\begin{align}
\label{eq3.4}
Y^x_t&=Y^x_{T\wedge\zeta^A}+\int^{T\wedge\zeta^A}_{t\wedge\zeta^A}dB_s\nonumber\\
&\quad+\int^{T\wedge\zeta^A}_{t\wedge\zeta^A}h(X^A_s)g(Y^x_s)\,dA^{\sigma}_s
-\int^{T\wedge\zeta^A}_{t\wedge\zeta^A}dM^x_s,\quad t\in[0,T].
\end{align}

We say that a c\`adl\`ag $(\FF^A_t)$-adapted process $Y$ is of class D under $P^A_x$ if
the collection $\{Y_{\tau}:\tau\mbox{ is a finite $(\FF^A_t)$-stopping time}\}$   is uniformly integrable under $P^A_x$.

\begin{definition}
Let $x\in\bar D$. We say that a pair $(Y^x,M^x)$ is a solution of (\ref{eq3.4}) if
\begin{enumerate}[(a)]
\item $Y^x$ is an $(\FF^A_t)$-progressively measurable c\`adl\`ag process, $Y^x$ is of class D under $P^A_x$ and $M^x$ is a c\`adl\`ag $(\FF_t)$-local martingale under $P^A_x$;
	
\item $E^A_x\int^{T\wedge\zeta^A}_0dB_s
+E^A_x\int^{T\wedge\zeta^A}_0h(X^A_t)g(Y^x_t)\,dA^{\sigma}_t<\infty$ $P^A_x$-a.s. for every $T>0$;

\item $Y^x_{T\wedge\zeta^A}\rightarrow0$ $P^A_x$-a.s. as $T\rightarrow\infty$ and for every $T>$ equation (\ref{eq3.4}) is satisfied $P^A_x$-a.s.
\end{enumerate}
\end{definition}

\begin{proposition}
\label{prop3.7}
Let  $u$ be a probabilistic solution of \mbox{\rm(\ref{eq1.1})}. Then there exists a process $M$ such that $M$ is a martingale under $P^A_x$ for q.e. $x\in\bar D$ and for q.e. $x\in\bar D$ the pair $(Y_t=u(X^A_t),M_t)$, $t\ge0$, is a solution of \mbox{\rm(\ref{eq3.4})} under $P^A_x$.
\end{proposition}
\begin{proof}
See \cite[Remark 3.2]{KR:NoD}.
\end{proof}

\subsection{Existence and uniqueness results}

In what follows we first prove the existence of probabilistic  solutions of (\ref{eq1.1}) for continuous nonincreasing $g$ satisfying (\ref{eq1.4}) and then for  general continuous $g$ satisfying (\ref{eq1.4}). In the first case the solution is unique. This follows from the following comparison result.

\begin{proposition}
\label{prop3.11}
Let $\mu_1,\mu_2$ be smooth measures on $D$ and $h_1,h_2:\partial D\rightarrow\BR$, $g_1,g_2:(0,\infty)\rightarrow\BR$ be nonnegative Borel measurable functions. Assume that  $\mu_1\le\mu_2$, $h_1\le h_2$,  $g_1(y)\le g_2(y)$ for $y>0$ and either $g_1$ or $g_2$ is nonincreasing. Assume also that for  $i=1,2$ there exists a probabilistic solution $u_i$ of \mbox{\rm(\ref{eq1.1})} with $\mu,h,g$ replaced by $\mu_i,h_i,g_i$, respectively.  Then $u_1\le u_2$ q.e. in $\bar D$.
\end{proposition}
\begin{proof}
Suppose that $g_2$ is nonincreasing. By Proposition \ref{prop3.7}, there exist processes $M^1,M^2$ such that for q.e. $x\in\bar D$ they are martingales under $P^A_x$ and for $i=1,2$ the pair
$(Y^i=u_i(X^A),M^i)$ is a solution of (\ref{eq3.4}) with $\mu,h,g$ replaced by $\mu_i,h_i,g_i$. Hence
\[
Y^i_t=Y^i_0-B^i_t-\int^{t}_0h_i(X^A_s)g_i(Y^i_s)\,dA^{\sigma}_s
+\int^t_0dM^i_s,\quad t\ge0,
\]
for $i=1,2$. By the  It\^o--Tanaka formula (see, e.g., \cite[Theorem IV.68]{P}),
\[
(Y^1_t-Y^2_t)^+-(Y^1_0-Y^2_0)^+\ge\int^t_0\fch_{\{Y^1_s-Y^2_s>0\}}\,d(Y^1_s-Y^2_s),\quad t\ge0.
\]
Since $B^1_t-B^2_t=A^{\mu_1-\mu_2}_t\le0$, $t\ge0$, it follows from the above inequality that for q.e. $x\in\bar D$,
\begin{equation}
\label{eq3.10}
(u_1-u_2)^+(x)\le E^A_x\int^{\infty}_0\fch_{\{u_1-u_2>0\}}(X^A_s)
(h_1g_1(u_1)-h_2g_2(u_2))(X^A_s)\,dA^{\sigma}_s.
\end{equation}
Since $h_1g_1(u_1)-h_2g_2(u_2)=(h_1g_1(u_1)-h_1g_1(u_1))+(h_1g_2(u_1)-h_2g_2(u_2))\le0$ if $u_1>u_2$, the righ-hand side of (\ref{eq3.10}) is nonpositive. Hence  $(u_1-u_2)^+=0$ q.e. which implies that $u_1\le u_2$ q.e. A similar argument applies in case $g_1$ is nonincreasing.
\end{proof}

For $k>0$ we define the  truncation operator by
\[
T_k(s)=\max(-k,\min(s,k)),\quad s\in\BR.
\]

\begin{theorem}
\label{th3.4}
Assume that \mbox{\rm(\ref{eq1.2})--(\ref{eq1.4})} are satisfied, and moreover $g$ is continuous and nonincreasing.  Then there exists a unique probabilistic  solution $u$ of \mbox{\rm(\ref{eq1.1})}.
\end{theorem}
\begin{proof}
Uniqueness follows from Proposition \ref{prop3.11}. To prove the existence of a solution,
for $n\ge1$ we define  $g_n:\BR\rightarrow[0,\infty)$ by $g_n(y)=g(y+1/n)$ if $y>0$ and $g_n(y)=g(1/n)$ if $y\le0$.  Clearly $g_n$ is nonincreasing, continuous and bounded.
By \cite[Theorem 4.3]{KR:PA} and Remark \ref{rem3.2}(iii), there exists a quasi continuous function $u_n$ and a c\`adl\`ag $(\FF^A_t)$ -adapted process $M^n$ such that $M^n$ is a   uniformly integrable martingale under $P^A_x$ for q.e. $x\in \bar D$ and   the pair $(Y^n,M^n)$ with $Y^n=u_n(X^A)$ is the unique solution of the BSDE
\begin{equation}
\label{eq3.12}
Y^n_t=\int^{\infty}_tdB_s+\int^{\infty}_th(X^A_s)g_n(Y^n_s)\,dA^{\sigma}_s
-\int^{\infty}_tdM^n_s,\quad t\ge0.
\end{equation}
Taking $t=0$ and intergrating with respect to $P^A_x$ shows that for q.e. $x\in\bar D$ we have
\begin{equation}
\label{eq3.1}
u_n(x)=E^A_xB_{\infty}+E^A_x\int^{\infty}_0h(X^A_t)g_n(u_n(X^A_t))\,dA^{\sigma}_t.
\end{equation}
Clearly $u_n\ge0$. From (\ref{eq1.2}) it follows that in fact $u_n>0$ q.e. for all sufficiently large $n\ge1$. Indeed, if  $\mu(D)>0$, then by Lemmas \ref{lem3.6} and \ref{lem3.4} we have
$E^A_xB_{\infty}=G^{\kappa}\mu(x)>0$ for q.e. $x\in \bar D$.
Suppose now that
$\|h\|_{L^1(\partial D;\sigma)}>0$. Let $g^1_n(y)=c_1(y+1/n)^{-\gamma}$ if  $y>0$ and $g^1_n(y)=c_1(1/n)^{-\gamma}$ if $y\le0$. Since $g^1$ satisfies the same assumptions as $g$, there is a quasi continuous function $u^1_n$ such that $u^1_n$ satisfies (\ref{eq3.1}) with $g_n$ replaced by $g^1_n$. By Proposition \ref{prop3.11}, $u_n\ge u^1_n$ q.e. On the other hand,
\[
u^1_n(x)\ge E_x\int^{\infty}_0h(X^A_t)g^1_n(u^1_n(X^A_t))\,dA^{\sigma}_t
=E_x\int^{\infty}_0\varphi(X^A_t)\,dA^{\sigma}_t
\]
with $\varphi(x)=c_1h(x)(u^1_n(x)+1/n)^{-\gamma}$. By (\ref{eq3.14}) and Lemma \ref{lem3.4},
$u^1_n(x)\ge G^{\kappa}(\varphi\cdot\sigma)(x)>0$
for q.e. $x\in\bar D$.
Since  $h\ge0$ and  $g_{n}\le g_{n+1}$, it follows from  Proposition \ref{prop3.11} that $u_n\le u_{n+1}$, $n\ge1$, q.e. Set $u=\limsup_{n\rightarrow\infty}u_n$. By what has already been proved,  $u\ge u^1_n>0$ q.e.
Define $w(x)=E^A_xB_{\infty}$ and  $Z^n_t=Y^n_t-\int^{\infty}_tdB_s$. Clearly $u_n\ge w$.
By  It\^o's formula,
\begin{align*}
(Z^n_t)^{\gamma+1}-(Z^n_0)^{\gamma+1}&=(\gamma+1)\int^t_0(Z^n_s)^{\gamma}\,dZ^n_s
+\frac12(\gamma+1)\gamma\int^t_0(Z^{n}_s)^{\gamma-1}\,d[Z^n]_s\\
&\ge-(\gamma+1)\int^t_0(Z^n_s)^{\gamma}h(X^A_s)g_n(Y^n_s)\,dA^{\sigma}_s  +(\gamma+1)\int^t_0(Z^n_s)^{\gamma}\,dM^n_s.
\end{align*}
Let $(\tau_k)$ be a localizing sequence for the local martingale $t\mapsto\int^t_0(Z^n_s)^{\gamma}\,dM^n_s$. By the above inequality we have
\begin{equation}
\label{eq3.20}
E(Z^n_0)^{\gamma+1}\le E(Z^n_{t\wedge\tau_k})^{\gamma+1}
+(\gamma+1)E\int^{t\wedge\tau_k}_0(Z^n_s)^{\gamma}h(X^A_s)g_n(Y^n_s)\,dA^{\sigma}_s.
\end{equation}
By (\ref{eq1.4}), $g_n(y)y^{\gamma}=g(y+1/n)y^{\gamma}\le g(y+1/n)(y+1/n)^{\gamma}\le c_2$. Therefore
letting $k\rightarrow\infty$ and then $t\rightarrow\infty$ in (\ref{eq3.20}) and using fact that  $E_x(Z^n_t)^{\gamma+1}\rightarrow0$ we get
\begin{align}
\label{eq3.19}
(u_n-w)^{\gamma+1}(x)&\le (\gamma+1)E^A_x\int^{\zeta^A}_0(u_n-w)(X^A_s))^{\gamma}
h(X^A_s)g_n(u_n(X^A_s))\,dA^{\sigma}_s\nonumber\\
	&\le c_2(\gamma+1)E^A_x\int^{\zeta^A}_0h(X^A_s)\,dA^{\sigma}_s.
\end{align}
Since  $\mu(D)<\infty$,  $w=G^{\kappa}\mu<\infty$ q.e. in $\bar D$ by Lemma \ref{lem3.6}. Therefore from (\ref{eq3.19}) it follows that
$\sup_{n\ge1}u^{\gamma+1}_n(x)< \infty$ for q.e. $x\in\bar D$.
Let $F_k=\{u_1\ge1/k\}$ and $\tau_k=\inf\{t>0: X^A_t\notin F_k\}$.
From (\ref{eq3.1}), by  using the strong Markov property and the fact that $B,A^{\sigma}$ are additive functionals of $\BM^A$, we get
\[
u_n(x)=E^A_xu_n(X^A_{\tau_k})+E^A_xB_{\tau_k}+E^A_x\int^{\tau_k}_0
h(X^A_t)g_n(u_n(X^A_t))\,dA^{\sigma}_t
\]
for all $n,k\ge1$. For every $x\in F_k$,
\[
g(u_n(x)+1/n)\le g(u_1(x)+1/n)\le c_2u_1^{-\gamma}(x)\le c_2k^{-\gamma},
\]
so letting $n\rightarrow\infty$ and applying the dominated convergence theorem we obtain
\[
E^A_x\int^{\tau_k}_0
h(X^A_t)g_n(u_n(X^A_t))\,dA^{\sigma}_t
\rightarrow E^A_x\int^{\tau_k}_0
h(X^A_t)g(u(X^A_t))\,dA^{\sigma}_t.
\]
Hence, for each $k\ge1$,
\begin{equation}
\label{eq3.6}
u(x)=E^A_xu(X^A_{\tau_k})+E^A_xB_{\tau_k}+E^A_x\int^{\tau_k}_0
h(X^A_t)g(u(X^A_t))\,dA^{\sigma}_t.
\end{equation}
From the definition of $\tau_k$ it follows that $u_1(X^A_{\tau_k})\le1/k$
and $\tau_k\nearrow\tau$ as $k\rightarrow\infty$ for some stopping time $\tau$. Furthermore, since $u_1$ is quasi continuous, we  have
\begin{equation}
\label{eq3.9}
P^A_x(0=\lim_{k\rightarrow\infty}u_1(X^A_{\tau_k})=u_1(X^A_{\tau}),\tau<\zeta^A)
=P^A_x(\tau<\zeta^A)
\end{equation}
for q.e. $x\in\bar D$.
We know that  $u_1>0$ q.e. Let $N\subset\bar D$ be a Borel $\EE^{\kappa}$-exceptional set such that $u_1>0$ on  $\bar D\setminus N$ and let $\sigma_N=\inf\{t>0:X^A_t\in N\}$.
Since $\mbox{Cap}^{\kappa}(N)=0$, where $\mbox{Cap}^{\kappa}$ is the capacity associated with $\EE^{\kappa}$,  there is a decreasing sequence of relatively compact open sets $(U_n)$ such that $N\subset\bigcap_{n\ge1}U_n$ and $\mbox{Cap}^{\kappa}(U_n)\downarrow0$. By \cite[Theorem 4.2.1]{FOT}, $P_x(\lim_{n\rightarrow\infty}\sigma_{U_n}=\infty)=1$ for  q.e. $x\in\bar D$, so $P_x(\sigma_N=\infty)=1$. 
Hence $P^A_x(\sigma_N<\zeta^A)=0$ for q.e. $x\in\bar D$.
On the other hand, by (\ref{eq3.9}) and the definition of $N$, for
$x\in\bar D$  we have
\[
P^A_x(\tau<\zeta^A)
=P^A_x(X^A_{\tau}\in N,\tau<\zeta^A)=P^A_x(\sigma_N\le\tau,\tau<\zeta^A),
\]
so $P^A_x(\tau<\zeta^A)=0$ for q.e. $x\in\bar D$. Consequently,
$P^A_x(\lim_{k\rightarrow\infty}\tau_k\ge\zeta^A)=1$ for q.e. $x\in\bar D$.
We next show that $Y:=u(X^A)$ is of class D. Indeed, by the strong Markov property,
for a finite $(\FF^A_t)$-stopping time $\sigma$ we have $w(X^A_{\sigma})=E^A_x(B_{\infty}\circ\theta_{\sigma}|\FF^A_{\sigma})=
E^A_x(B_{\infty}-B_{\sigma}|\FF^A_{\sigma})\le E^A_x(B_{\infty}|\FF^A_{\sigma})$, which together with the fact that $w<\infty$ q.e. in $\bar D$ implies that for q.e. $x\in\bar D$ the process $w(X^A)$ is of class D under $P^A_x$. Using (\ref{eq3.19}) and the strong Markov property we also show that
\[
|(u-w)(X^A_{\sigma})|^{\gamma+1}\le c_2(\gamma+1)E^A_x\Big(\int^{\zeta^A}_0h(X^A_s)\,dA^{\sigma}_s|\FF^A_{\sigma}\Big).
\]
Hence  $\sup_{\sigma}E^A_x|u(X^A_{\sigma})|^{\gamma+1}<\infty$, which implies that
the process $(u-w)(X^A)$ is of class D under $P^A_x$.
Consequently, for q.e. $x\in\bar D$, the process $u(X^A)$ is of class D under $P^A_x$. As a result,   $E^A_xu(X^A_{\tau_k})\rightarrow0$  for q.e. $x\in\bar D$ since $u(X^A_{\tau_k})\rightarrow0$ $P_x$-a.s.
Therefore letting $k\rightarrow\infty$ in (\ref{eq3.6}) shows that  $u$ satisfies (\ref{eq3.2}), so  $u$ is a probabilistic solution of (\ref{eq1.1}).
\end{proof}

We now turn to (\ref{eq1.1}) with general continuous $g$ satisfying (\ref{eq1.4}). It is convenient to consider first the problem  with bounded $h$ and then the general case.

\begin{proposition}
\label{prop3.5}
Assume that  $\mu$ satisfies \mbox{\rm(\ref{eq1.2})}, $h$ is   bounded, $g$ is continuous and bounded. Then there exists   a probabilistic solution of  \mbox{\rm(\ref{eq1.1})}.
\end{proposition}
\begin{proof}
Our proof is a modification of the proof of \cite[Proposition 3.7]{K:JFA}. By \cite[Lemma 6.1.2]{FOT}, the measure  $\sigma$ is smooth relative to $\EE^{\kappa}$. Since this form is transient, by the 0-order version of \cite[Theorem 2.2.4]{FOT} (see the remarks following \cite[Corollary 2.2.2]{FOT}) there exists a nest $(F_k)$ in $\bar D$ such
such that $\mu_k=\fch_{F_k}\cdot\mu$, $\sigma_k:=\fch_{F_k}\cdot\sigma\in S^{(0),\kappa}_0$ and $\|U^{\kappa}\mu_k\|_{\infty}+\|U^{\kappa}\sigma_k\|_{\infty}<\infty$  for each $k$.
By Lemma \ref{lem3.5}, $\|G^{\kappa}\mu_k\|_{\infty}+\|G^{\kappa}\sigma_k\|_{\infty}<\infty$, $k\ge1$. We will show that there exists $u_k$ such that for q.e. $x\in\bar D$ we have
\begin{equation}
\label{eq3.16}
u_k(x)=G^{\kappa}\mu_k(x)+G^{\kappa}(hg(u_k)\cdot\sigma_k)(x).
\end{equation}
To this end, we define $\Phi:H^1(D)\rightarrow H^1(D)$ by $\Phi(u)=G^{\kappa}\mu_k+G^{\kappa}(h(g(u)\cdot\sigma_k))$. Then $\|\Phi(u)\|_{\infty}\le\|G^{\kappa}\mu_k\|_{\infty}
+\|h\|_{\infty}\|g\|_{\infty}\|G^{\kappa}\sigma_k\|_{\infty}:=M$. Consider the subset $K$ of $H^1(D)$ defined by $K=\{u\in H^1(D):\|u\|_{\infty}\le M\}$. It is closed,  convex and
$\Phi:K\rightarrow K$. Since the usual norm in $H^1(D)$ is equivalent to the norm $\|\cdot\|_{\kappa}$, in much the same way as in the proof of \cite[Proposition 3.7]{K:JFA} (we apply the arguments to $\EE^{\kappa}$) we show that $\Phi$ is continuous
and that for any sequence in $\Phi(K)$  there is a subsequence converging in $H^1(D)$, i.e. $\Phi(K)$ is relatively compact in $H^1(D)$. Therefore we can apply Schauder's fixed point theorem to get a fixed point $u_k$ of $\Phi$, which is a solution of (\ref{eq3.16}).
By  (\ref{eq2.4}) and (\ref{eq3.14}),  for q.e. $x\in\bar D$,
\begin{equation}
\label{eq3.11}
u_k(x)=E^A\int^{\infty}_0\fch_{F_k}(X^A_t)\,dB_t
+E^A_x\int^{\infty}_0g(u_k(X^A_t))\fch_{F_k}(X^A_t)\,dA^{\sigma}_t.
\end{equation}
Since $\{F_k\}$ is a nest, $\fch_{F_k}(X^A_t)\rightarrow1$, $t\in[0,\zeta^A)$, $P_x$-a.s. for q.e $x\in \bar D$. Furthermore, by the argument at the and of the proof of \cite[Proposition 3.7]{K:JFA}, there is a subsequence, still denoted by $(u_k)$, such that $(u_k)$ converges q.e. to some $u$. Letting $k\rightarrow\infty$ in (\ref{eq3.11}) and
using the Lebesgue dominated convergence theorem shows that $u$ is an integral solution of (\ref{eq1.1}) and hence a probabilistic solution by Proposition \ref{prop3.3}.
\end{proof}

\begin{theorem}
\label{th3.6}
Assume that \mbox{\rm(\ref{eq1.2})--(\ref{eq1.4})} are satisfied and $g$ is continuous. Then there exists a probabilistic solution $u$ of \mbox{\rm(\ref{eq1.1})} such that $hg(u)\in L^1(\partial D;\sigma)$.
\end{theorem}
\begin{proof}
Define $h_n=T_n(h)$ and $g_n$ as in the proof of Theorem \ref{th3.4}. By Proposition \ref{prop3.5}, for every $n\ge1$ there exists a solution $u_n$	of the problem
\begin{equation}
\label{eq3.5}
-\Delta u_n=\mu\quad\mbox{in }D,\qquad-\frac{\partial u_n}{\partial{\mathbf{n}}}
+\beta\cdot u_n=h_n(x)g_n(u_n)\quad \mbox{on }\partial D.
\end{equation}
Let
$g^1(y)=c_1y^{-\gamma}$, $g^2=c_2y^{-\gamma}$, $y>0$ and
$g^1_n(y)=g^1(y+1/n)$, $g^2_n(y)=g^2(y+1/n)$, $y>0$. By Proposition \ref{prop3.5} and Theorem \ref{th3.4}, there exists  a unique probabilistic solution $u^1_n$ of (\ref{eq3.5}) with $g_n$ replaced by $g^1_n$, and similarly, there exists  a unique solution $u^2_n$ of (\ref{eq3.5}) with $g_n$ replaced by $g^2_n$. Furthermore, since $\mu(D)+\|h_n\|_{L^1(\partial D;\sigma)}>0$ for all sufficiently large $n$, we have $u^1_n>0$, $u^2_n>0$ q.e. for all suficiently large $n$. By  Proposition \ref{prop3.11} we also have $u^1_n\le u_n\le u^2_n$, $n\ge1$, q.e.
It follows in particular that
\begin{equation}
\label{eq3.7}
g(u_n+1/n)\le c_2(u_n+1/n)^{-\gamma}\le (c_2/c_1)g^1_n(u^1_n),\quad n\ge1,\quad\mbox{q.e.}
\end{equation}
Let $u^1,u^2$ be the unique probabilistic solutions of problems (\ref{eq1.1}) with $g$ replaced by $g^1$ and $g$ replaced by $g^2$, respectively.
From the proof of Theorem \ref{th3.4} it follows that $u^1_n\rightarrow u^1$ q.e. Hence
\begin{equation}
\label{eq3.8}
h_n(X^A)g^1_n(u^1_n(X^A))\rightarrow h(X^A)g^1(u^1(X^A))\quad P^A_x\otimes dA^{\sigma}\mbox{-a.s.}
\end{equation}
for q.e. $x\in \bar D$, where $P^A_x\otimes dA^{\sigma}$ is the product of the measure $P^A_x$ and the kernel $dA^{\sigma}$ from $\Omega$ to $\BB(\BR_+)$.  Moreover,
\[
u^1_n(x)=E^A_xB_{\infty}+E^A_x\int^{\infty}_0
h_n(X^A_t)g^1_n(u^1_n(X^A_t))\,dA^{\sigma}_t
\rightarrow u^1(x),
\]
which together with (\ref{eq3.8}) implies that  the family $\{h_n(X^A)g^1_n(u^1_n(X^A))\}$ is uniformly integrable  with respect to the measure $P^A_x\otimes dA^{\sigma}$ for q.e. $x\in \bar D$. By this and (\ref{eq3.7}), $\{h_n(X^A)g_n(u_n(X^A))\}$ is uniformly integrable. Consequently, for q.e. $x\in \bar D$ we have
\[
\lim_{t\rightarrow0^+}\sup_{n\ge1}
E_x\int^t_0h_n(X^A_s)g_n(u_n(X^A_s))\,dA^{\sigma}_s=0.
\]
By the Markov property,
\[
u_n(x)=E^A_xB_t+E^A_x\int^t_0h_n(X^A_s) g_n(u_n(X^A_s))\,dA^{\sigma}_s+P^A_tu_n(x),\quad t\ge0,
\]
Hence, for q.e. $x\in\bar D$,
\[
\lim_{t\rightarrow t^+}\sup_{n\ge1}|u_n(x)-P^A_tu_n(x)|=0,
\]
which means that condition ($\mbox{M}_1$) of  \cite{K:PA} is satisfied. Let $\BB_1=\{v\in\BB^+(\bar D):v(x)\le1,x\in\bar D\}$. By \cite[Proposition 2.4]{K:PA}, the triple  $(\BM^A,\BB_1,m)$ has the compactness property, so by \cite[Proposition 4.2]{K:PA}, $(\BM^A,[0,u^2],m)$ has the compactness property. Since $u_n\le u^2_n\le u^2$ for $n\ge1$, applying \cite[Theorem 2.2]{K:PA} shows that there is a subsequence (still denoted by $(u_n)$) such that $(u_n)$ converges q.e. as $n\rightarrow\infty$. As in the proof of Theorem \ref{th3.4} we show that its limit $u$ is a probabilistic solution of (\ref{eq1.1}).
We next show that $hg(u)\in L^1(\partial D;\sigma)$.
Let $Y^{n,1}_t=u^1_n(X^A_t)$, $t\ge0$. Then $Y^{n,1}$ is the first component of the solution $(Y^{n,1},M^{n,1})$ of (\ref{eq3.12}) with $g_n$ replaced by $g^1_n$ and $h$ replaced by $h_n$.  By \cite[Theorem IV.68]{P}, for every $t\ge0$,
\begin{align*}
E_x(Y^{n,1}_t-1)^--E_x(Y^{n,1}_0-1)^-
&\ge- E_x\int^t_0\fch_{\{Y^{n,1}_s\le 1\}}\,dY^{n,1}_s\\
&=E_x\int^t_0\fch_{\{Y^{n,1}_s\le 1\}} \{dB_s+h_n(X^A_s)g^1_n(Y^{n,1}_s)\,dA^{\sigma}_s\}.
\end{align*}
Hence
\[
E_x\int^t_0\fch_{\{Y^{n,1}_s\le 1\}}h_n(X^A_s)g_n(Y^{n,1}_s)\,dA^{\sigma}_s\le E_x(Y^{n,1}_t-1)^-,\quad t\ge0.
\]
Letting $t\rightarrow\infty$ yields
\begin{equation}
\label{eq3.13}
E_x\int^{\infty}_0\fch_{\{Y^{n,1}_s\le 1\}} h_n(X^A_s)g^1_n(Y^{n,1}_s)\,dA^{\sigma}_s\le 1.
\end{equation}
Set $\nu_n=\fch_{\{u^1_n\le1\}}h_ng^1_n(u^1_n)\cdot\sigma$. Then  $\nu_n\in S^{(0),\kappa}_0$ since $h_ng^1_n(u^1_n)$ is bounded. Furthermore,
since $1\in H^1(D)$, by (\ref{eq2.5}) we have
$\EE^{\kappa}(U^{\kappa}\nu_n,1)=\langle\nu_n,1\rangle$.
By Lemma \ref{lem3.5}, $G^{\kappa}\nu_n$ is a quasi continuous version of $U^{\kappa}\nu_n$, and  by Schwartz's inequality, $\EE(U^{\kappa}\nu_n,1)=0$ since
$\EE(1,1)=0$. Therefore
\[
\langle\nu_n,1\rangle=\EE^{\kappa}(U^{\kappa}\nu_n,1)=\langle\kappa,G^{\kappa}\nu_n\rangle.
\]
But $G^{\kappa}\nu_n=G^{\kappa}(\varphi_n\cdot\sigma)$ with
$\varphi_n=\fch_{\{u^1_n\le 1\}}h_ng^1_n(u^1_n)$, so   by (\ref{eq3.14})
and (\ref{eq3.13}),
\[
G^{\kappa}\nu_n(x)=E^A_x\int^{\infty}_0 \fch_{\{Y^{n,1}_s\le 1\}} h_n(X^A_s)g^1_n(Y^{n,1}_s)\,dA^{\sigma}_s\le 1.
\]
As a result we obtain
\[
\nu_n(\bar D)\le \kappa(\bar D)=\|\beta\|_{L^1(\partial D;\sigma)}.
\]
Since $u^1_n\nearrow u^1$, letting $n\rightarrow\infty$ and using Fatou's lemma
shows that $\fch_{\{u^1\le 1\}}hg^1(u^1)\in L^1(\partial D;\sigma)$, hence that $hg^1(u^1)\in L^1(\partial D;\sigma)$. From this, inequalities (\ref{eq3.7})
and the fact that $u_n\rightarrow u$ it follows that
$hg(u)\in L^1(\partial D;\sigma)$.
\end{proof}

\subsection{Equations with mixed nonlinearities}

We consider the problem
\begin{equation}
\label{eq3.15}
-L u=\mu\quad\mbox{in }D,\qquad-\frac{\partial u}{\partial{\gamma_a}}
+\beta\cdot u=h(g^1(u)+g^2(u))\quad \mbox{on }\partial D,
\end{equation}
with  $g^1,g^2$ satisfying (\ref{eq1.7}). The proof of the theorem below is a slight modification of the proof of \cite[Theorem 3.11]{K:JFA}.

\begin{theorem}
Assume that $\mu,h$ satisfy \mbox{\rm(\ref{eq1.2})}, $g^1,g^2:(0,\infty)\rightarrow[0,\infty)$ are continuous nonincreasing functions satisfying \mbox{\rm(\ref{eq1.7})}. Then there exsists a unique probabilistic solution of problem \mbox{\rm(\ref{eq3.15})}.
\end{theorem}
\begin{proof}
Set $h_n=T_n(h)$, $g^1_n(y)=g(y+1/n)$, $g^2_n(y)=g^2(y+1/n)$, $y>0$. By Propositions \ref{eq3.11} and \ref{prop3.5}, there exists a unique probabilistic solution $u_n$ of  problem (\ref{eq3.15}) with $h,g^1,g^2$ replaced by $h_n,g^1_n,g^2_n$, i.e.
\begin{equation}
\label{eq3.18}
u_n(x)=E^A_xB_{\infty}+E^A_x\int^{\infty}_0
(h_n(u^1_n)(X^A_t)+h_n(u^2_n(X^A_t)))\,dA^{\sigma}_t
\end{equation}
for q.e. $x\in\bar D$. By Proposition \ref{prop3.11}, $u_n\le u_{n+1}$, $n\ge1$, q.e. Since $c_1y^{-\gamma_1}\le(g^1+g^2)(y)$ and  $c_1y^{-\gamma_2}\le(g^1+g^2)(y)$, applying Proposition \ref{prop3.11} we also get
\begin{equation}
\label{eq3.17}
v_n\le u_n,\quad w_n\le u_n,\quad n\ge1,\quad\mbox{q.e.}
\end{equation}
where $u^i_n$, $i=1,2$,  is the solution of the problem
\[
(-L+\beta\cdot\sigma)u^i_n=\mu+c_1h_n(u^i_n+1/n)^{-\gamma_i}.
\]
By (\ref{eq3.17}), $u^1_n+u^2_n\le 2u_n$, $n\ge1$, q.e. As in the proof of \cite[Theorem 3.11]{K:JFA},  from this inequality and (\ref{eq1.7}) we deduce that
\begin{equation}
\label{eq3.3}
g^1(u_n+1/n)+g^2(u_n+1/n)\le c_22^{\gamma_1}(u^1_n+1/n)^{-\gamma_1}+c_22^{\gamma_2}(u^2_n+1/n)^{-\gamma_2}.
\end{equation}
The argument following (\ref{eq3.8}) shows that  for $i=1,2$, the sequences $\{(u^i_n+1/n)^{-\gamma_i}\}$ are uniformly integrable with respect to the measure $P^A_x\otimes dA^{\sigma}$, which together with (\ref{eq3.3}) implies that $\{(h_ng^1_n(u^1_n)+h_ng^2_n(u^2_n))(X^A)\}$ is uniformly integrable with respect to $P^A_x\otimes dA^{\sigma}$. Therefore letting $n\rightarrow\infty$ in (\ref{eq3.18}) shows that $u$ defined by $u=\lim_{n\rightarrow\infty}u_n$ is a probabilistic solution of (\ref{eq3.15}).
\end{proof}

\section{Renormalized and weak solutions}

As in previous sections, we assume that $\mu$ is a smooth measure on $D$,   $h, g,\beta$ are nonnegative Borel measurable functions,  and we write $\kappa=\beta\cdot\sigma$.

\subsection{Renormalized solutions}

We begin with renormalized solutions. The following definition is an adaptation of \cite[Definition 3.4]{KR:NoD} to  problem (\ref{eq1.1}) which actually  is regarded as  problem (\ref{eq1.5}) with $\Delta_N$ replaced by the operator associated with the form $\EE$.

Let $\nu$ be a signed Borel measure on $\bar D$, and let $\nu^+$ and $\nu^-$ denote its  positive and negative part, respectively. We call it smooth if $|\nu|=\nu^++\nu^-$ is smooth.  We denote by  $\MM_{0,b}(\bar D)$ the set of all smooth signed  measures $\nu$ on $\bar D$, with respect to the form $\EE$, such that $\|\nu\|_{TV}:=|\nu|(\bar D)<\infty$. By \cite[Lemma 6.1.2]{FOT}, $\nu$ is smooth with respect to $\EE$ if and only if $\nu$  is smooth with respect to $\EE^{\kappa}$.

\begin{definition}
Assume that $\mu(D)<\infty$ and  $h\in L^{\infty}(\partial D;\sigma)$. A quasi continuous function $u$ defined q.e on $\bar D$ is a renormalized solution of (\ref{eq1.1}) if

\begin{enumerate}
\item[(a)] $u$ is quasi continuous, $u>0$ q.e. in $\bar D$, $hg(u)\in L^1(\partial D;\sigma)$ and $T_k(u)\in H^1(D)$ for every $k>0$;

\item[(b)] there exists a sequence $\{\nu_k\}\subset\MM_{0,b}(\bar D)$ such that $\lim_{k\rightarrow\infty}\|\nu_k\|_{TV}=0$ and for every $k>0$
	and every bounded quasi continuous version $\tilde v$ of  $v\in H^1(D)$,
\[
\EE^{\kappa}(T_k(u),v) =\langle \mu+hg(u)\cdot\sigma,\tilde
v\rangle +\langle\nu_{k},\tilde v\rangle.
\]
\end{enumerate}

Note that by \cite[Theorem 2.1.3]{FOT}, each $v\in H^1(D)$ has a quasi continuous modification.
\end{definition}

\begin{theorem}
\label{th4.2}
Under the assumptions of Theorem \ref{th3.6} there exists a renormalized solution of \mbox{\rm(\ref{eq1.1})} such that $hg(u)\in L^1(\partial D;\sigma)$. If moreover $g$ is nonincreasing, then there is a unique renormalized solution   and it is the unique probabilistic solution.
\end{theorem}
\begin{proof}
From Theorem \ref{th3.6} we know that there exists a probabilistic solution $u$ such that $hg(u)\in L^1(\partial D;\sigma)$. Since $\mu,\sigma$ are smooth bounded measures,  $\mu+hg(u)\cdot\sigma\in\MM_{0,b}(\bar D)$. Therefore $u$ is a renormalized solution by
\cite[Theorem 3.5]{KR:NoD} applied to the form $\EE^{\kappa}$. The second part follows from the first one, Theorem \ref{th3.4} and \cite[Theorem 3.5]{KR:NoD} again.
\end{proof}

The following regularity results follow directly from general results for equations with smooth measure data and operators associated with transient symmetric Dirichlet forms.

\begin{proposition}
\label{prop4.3}
Let the assumptions of Theorem \ref{th3.6} hold and let $u$ be a renormalized solution of \mbox{\rm(\ref{eq1.1})} such that $hg(u)\in L^1(\partial D;\sigma)$. Then,
for every $k>0$, $T_k(u)\in H^1(D)$ and \mbox{\rm(\ref{eq1.8})} is satisfied.
If, in addition, $\mu=0$, then $u^{(\gamma+1)/2}\in H^1(D)$ and \mbox{\rm(\ref{eq1.9})} is satisfied for some constant $c(\gamma)>0$.
\end{proposition}
\begin{proof}
The first part follows from \cite[Proposition 5.9]{KR:JFA} and the second from \cite[Theorem 4.6]{K:JFA}.
\end{proof}

For $r>0$ we denote by $M^r(D;m)$ the Marcinkiewicz space of order $r$. It is the set of  measurable functions $v$ on $D$ such that the distribution function corresponding to  $m$, that is
the function $\lambda_m(t)=m(\{x\in D: |v(x)|\ge t\})$, $t>0$, satisfies  estimate of the form $\lambda_m(t)\le Ct^{-r}$ for some $C>0$. Similarly, by using the distribution function corresponding  to $\sigma$, we define the Marcinkiewicz space $M^r(\partial D;\sigma)$.

\begin{corollary}
Under the assumptions of Theorem \ref{th3.6},  $u$ of Proposition \ref{prop4.3} has the property that
$u\in M^{(d-1)/(d-2)}(\partial D;\sigma)$ and
$u\in M^{d/(d-2)}(D;m)$, $|\nabla u|\in M^{d/(d-1)}(D;m)$.
\end{corollary}
\begin{proof}
Follows from (\ref{eq1.8}), (\ref{eq2.1}) and known results (see \cite[Lemmas 4.1, 4.2]{BBGGPV} and  \cite[p. 11]{DOS}).
\end{proof}

From Theorem \ref{th4.2} we know that for nonincreasing $g$ renormalized solutions coincide with  probabilistic solutions. As a simple  application of this fact and Proposition \ref{prop3.7} we get the following stability result.

\begin{proposition}
Assume that  $\mu=f\cdot m$, $\mu_n=f_n\cdot m$, $n\ge1$, for some $f,f_n\in L^1(D;m)$. Assume also that
$\beta,h,g$ satisfy the assumptions of Theorem \ref{th3.4} and $\|h\|_{L^1(\partial D;\sigma)}>0$. Let $u$ be the unique renormalized solution of \mbox{\rm(\ref{eq1.1})} and $u_n$ be the renormalized solution of \mbox{\rm(\ref{eq1.1})} with $\mu$ replaced by $\mu_n$. 	If $f_n\rightarrow f $ in $L^1(D;m)$ and $f_n\le F$ for some $F\in L^1(D;m)$. Then $u_n\rightarrow u$ q.e. in $\bar D$.
\end{proposition}
\begin{proof}	
By Proposition \ref{prop3.7}, for each $n\ge1$ there exists a process $M^n$ such that $M^n$ is a martingale under $P^A_x$ for q.e. $x\in\bar D$ and  the pair $(Y^n=u_n(X^A),M^n)$ is the solution of (\ref{eq3.4}) with $B$ replaced by $B^n_t=\int^t_0f_n(X^A_s)\,ds$, $t\ge0$. Hence
\[
Y^n_t=Y^n_0-\int^t_0f_n(X^A_s)\,ds-\int^t_0h(X^A_s)g(Y^n_s)\,dA^{\sigma}_s
+\int^t_0dM^n_s,\quad t\ge0.
\]	
Similarly, 	for $(Y=u(X^A),M)$ of Proposition \ref{prop3.7} we have
\[
Y_t=Y_0-\int^t_0f(X^A_s)\,ds-\int^t_0h(X^A_s)g(Y_s)\,dA^{\sigma}_s
+\int^t_0dM_s,\quad t\ge0.
\]	
Applying the It\^o--Tanaka formula we get
\[
|Y^n_t-Y_t|-|Y^n_0-Y_0|\ge\int^t_0\mbox{sgn}(Y^n_s-Y_s)\,d(Y^n_s-Y_s),\quad t\ge0,\quad P_x\mbox{-a.s.}
\]
for q.e. $x\in \bar D$. Hence
\begin{align*}
E_x|Y^n_0-Y_0|&\le E^A_x\int^{\infty}_0\mbox{sgn}(Y^n_s-Y_s)(f_n(X^A_s)-f(X^A_s))\,ds\\
&\quad+E^A_x\int^{\infty}_0\mbox{sgn}(Y^n_s-Y_s)
h(X^A_s)(g(Y^n_s)-g(Y_s))\,dA^{\sigma}_t.
\end{align*}
The second integral on the rigt-hand side of the above inequality is less then or equal to zero. Therefore
\begin{equation}
\label{eq4.2}
|u_n(x)-u(x)|\le E^A_x\int^{\infty}_0|f_n-f|(X^A_s)\,ds=G^{\kappa}|f_n-f|(x).
\end{equation}		
For q.e. $x\in\bar D$, $G^{\kappa}(x,y)|f_n-f|(y)\le 2G^{\kappa}(x,y)F(y)$, $y\in D$, and,
since $F\in L^1(D;m)$, $G^{\kappa}F(x)=R^AF(x)<\infty$ for q.e. $x\in\bar D$. Hence,  for   q.e. $x\in\bar D$, the function $D\ni y\mapsto G^{\kappa}(x,y)F(y)$ is integrable. Therefore letting $n\rightarrow\infty$ in (\ref{eq4.2}) and applying the dominated convergence theorem shows that $u_n(x)\rightarrow u(x)$ for q.e. $x\in\bar D$.
\end{proof}

\subsection{Weak solutions}

\begin{definition}
A function $u\in H^1(D)$ is a weak solution of (\ref{eq1.1}) if $u>0$ $m$-a.e. and
\begin{equation}
\label{eq4.1}
\EE^{\kappa}(u,v)=\langle\mu,v\rangle+\int_{\partial D}hg(u)v\,d\sigma,
\quad v\in H^1(D)\cap L^{\infty}(\partial D;\sigma).
\end{equation}
\end{definition}

\begin{proposition}
\label{prop4.5}
If $u$ is a probabilistic solution of \mbox{\rm(\ref{eq1.1})} such that $\nu\in S^{(0),\kappa}_0$, where $\nu=\mu+hg(u)\cdot\sigma$,  then $u$ is a weak solution.
\end{proposition}
\begin{proof}
By Proposition \ref{prop3.3}, $u=G^{\kappa}\nu$. Since $\nu\in S^{(0),\kappa}_0$,  the function $G^{\kappa}\nu$ is an $m$-version of the potential $U^{\kappa}\nu$.
Therefore $u\in H^1(D)$ and, by (\ref{eq2.5}), it satisfies (\ref{eq4.1}).
\end{proof}

Unfortunately, we are able to check that $\nu$ of Proposition \ref{prop4.5} is of class
$S^{(0),\kappa}_0$ only in some very special situations.

In the remainder of this subsection, we assume that $\mu(dx)=f\cdot m$ for some nonnegative Borel function on $D$.

\begin{proposition}
\label{prop4.6}
Let  the assumptions of Theorem \ref{th3.6} hold and let $u$ be a solution such that $u\ge\delta$ for some constant $\delta>0$.  If, in addition,
 $f\in L^p(D;m)$ with $p=2d/(d+2)$ and $h\in L^q(\partial D;\sigma)$ with $q=
2(d-1)/d$, then $u$ is a weak solution of \mbox{\rm(\ref{eq1.1})}.
\end{proposition}
\begin{proof}
By Proposition \ref{prop4.5}, it suffices to check that $f\cdot m,hg(u)\cdot\sigma $ are measures of finite 0-energy  integral. Write $p'=2d/(d-2)$, $q'=2(d-1)/(d-2)$.
By H\"older's inequality  and Sobolev's inequality, for $v\in H^1(D)\cap C(\bar D)$ we have
\[
\int_D|v|f\,dx\le \|f\|_{L^{p}(D;m)}\|v\|_{L^{p'}(D;m)},\quad
\|v\|_{L^{p'}(D;m)}\le c\|v\|_{H^1(D)}.
\]
Hence $f\cdot m\in S^{(0),\kappa}_0$. Furthermore, by H\"older's inequality and the  trace theorem (see, e.g., \cite[Chapter 2, Theorem 4.2]{N}, 
\[
\int_{\partial D}|v|h\,d\sigma\le
\|h\|_{L^{q}(\partial D;\sigma)}\|v\|_{L^{q'}(\partial D;\sigma)},\qquad \|v\|_{L^{q'}(\partial D;\sigma)}\le c\|v\|_{H^1(D)},
\]
which implies that $g(u)\cdot\sigma\in S^{(0),\kappa}_0$ since $g(u)\le c_2\delta^{-\gamma}$.
\end{proof}

\begin{remark}
\label{rem4.8}
The argument used in Proposition \ref{prop4.6}  to prove that $g(u)\cdot\sigma\in S^{(0),\kappa}_0$ does not take into account assumption (\ref{eq1.4}). Using (\ref{eq1.4}) one can show the existence of a weak solution under slightly weaker  integrability assumption imposed on $h$. Namely, as observed in \cite{DOS}, it suffices to assume that $h\in L^r(\partial D;\sigma)$ with $r=2(d-1)/(d+\gamma(d-2))$.
\end{remark}

We shall see that the condition  $u\ge\delta$ appearing in Proposition \ref{prop4.6} is satisfied in the case where $\BM^A$ strong Feller. A known probabilistic criterion for this to hold is given below.

\begin{lemma}
\label{lem4.9}
If
\begin{equation}
	\label{eq4.7}
	\lim_{t\downarrow0}\sup_{x\in\bar D}E_x(1-e^{-A_t})=0,
\end{equation}
then the semigroup $(P^A_t)_{t>0}$ is strong Feller.
\end{lemma}
\begin{proof}
For any bounded $f\in\BB^+(\bar D)$ and $s>0$,
\[
\sup_{x\in\bar D}|P^A_sf(x)-P_sf(x)|\le\|f\|_{\infty}\sup_{x\in\bar D}E_x(1-e^{-A_s}).
\]
Furthermore, for any $0\le s<t$, $P_sP^A_{t-s}f$ is continuous  (since $\BM$ is strong Feller) and
\begin{align*}
	\sup_{x\in\bar D}|P^A_tf(x)-P_sP^A_{t-s}f(x)|&=\sup_{x\in\bar D} |P^A_sP^A_{t-s}f(x)-P_sP^A_{t-s}f(x)|\\
	&\le\|f\|_{\infty}\sup_{x\in\bar D}|P^A_sf(x)-P_sf(x)|.
\end{align*}
from which the results follows.
\end{proof}

\begin{proposition}
\label{prop4.9}
Let \mbox{\rm(\ref{eq4.7})} hold. If $f\in\BB^+(D)$ and  $\|f\|_{L^1(D;m)}>0$, then there is $\delta>0$ such that  $R^Af(x)\ge\delta$ for all $x\in\bar D$.
\end{proposition}
\begin{proof}
Choose $k>0$ so that $\|f_k\|_{L^1(D;m)}>0$, where $f_k=f\wedge k$. By Lemma \ref{lem4.9}, $P^A_tf_k\in C_b(\bar D)$ for $t>0$, so
$R^A_1f_k$ is continuous on $\bar D$, and hence attains its minimum, say $\delta$, in $D$. On the other hand, from  Lemma \ref{lem3.4} it follows that  $R^A_1f_k>0$ q.e. in $\bar D$. Thus, in fact, $R^A_1f_k(x)\ge\delta$ for $x\in\bar D$. This proves the proposition since obviously $R^Af\ge R^A_1f_k\ge R^A_1f_k$.
\end{proof}

\begin{corollary}
\label{cor4.11}
Let $\beta$ satisfy  \mbox{\rm(\ref{eq1.3})},    $f\in\BB^+(D)$ and  $\|f\|_{L^1(D;m)}>0$, $f\in L^p(D;m)$ with $p=2d/(d+2)$,  $h\in L^r(\partial D;\sigma)$ with $r=2(d-1)/(d+\gamma(d-2))$ and $g$ be a continuous function satisfying \mbox{\rm(\ref{eq1.4})}. If $a$ is smooth, then  there exists a weak solution of \mbox{\rm(\ref{eq1.1})}.
\end{corollary}
\begin{proof}
From  \cite[Proposition 6.1]{M} it follows that  (\ref{eq4.7}) is satisfied for bounded Lipschitz domains and smooth $a$ satisfying (\ref{eq2.7}). Therefore the desired result follows from Propositions \ref{prop4.6}, \ref{prop4.9}  and Remark \ref{rem4.8}.
\end{proof}

By applying Corollary \ref{cor4.11} to  $a=I_d$ (i.e. to $L=\Delta$) we get \cite[Theorem 2.2]{DOS}.

\end{document}